\documentclass[letterpaper, 10 pt, conference]{ieeeconf}  

\IEEEoverridecommandlockouts                              

\pdfoutput=1

\usepackage{graphicx} 
\usepackage{amsmath} 
\usepackage{amssymb}  
\usepackage{subcaption}
\usepackage{algorithm}
\usepackage{algpseudocode}

\DeclareMathOperator*{\argmin}{arg\,min}

\newtheorem{proposition}{Proposition}

\title{\LARGE \bf On the Optimal Policies for Visibility-Based Target Tracking}

\author{Rui Zou, Hamid Emadi and Sourabh Bhattacharya
	\thanks{Rui Zou, Hamid Emadi and Sourabh Bhattacharya are with the Department of Mechanical Engineering, Iowa State University, Ames, IA 50011, USA
		{\tt\small \{rzou, emadi, sbhattac\}@iastate.edu}}%
}

\begin{document}

\maketitle
\thispagestyle{empty}
\pagestyle{empty}

\begin{abstract}

In this paper, we investigate a pursuit-evasion game in which a mobile observer tries to track a target in an environment containing obstacles. We formulate the game as an optimal control problem with state inequality constraint in a simple environment. We show that for some initial conditions, there are two different regimes in the optimal strategy of the pursuer depending on whether the state-constraint is activated. We derive the equations that characterize the switching time between the two regimes. The pursuer's optimal tracking strategy in a simple environment is further extended to a general environment with multiple polygonal obstacles. We propose techniques to construct a ``pursuit field" based on the optimal solutions to guide the motion of the observer in a general environment.

\end{abstract}

\section{Introduction}

Surveillance of mobile targets is a problem that arises in numerous applications. For example, a missile defense system must be able to detect and track suspicious targets before intercepting the targets. Companies, organizations and even homes use CCTV camera systems to monitor restricted regions for suspicious activities. Traffic cameras are widely used in assisting traffic control, detecting accidents and catching criminals. However, the problem of ``data deluge'' arises for many visual surveillance systems. For example, a surveillance system for detecting intruders will record or transmit the same amount of data whether the scene is active or not. To alleviate the data overload, the authors have proposed a new paradigm in their previous work \cite{bhattacharya2014vision,Warnell2015} called ``opportunistic sensing''. The idea is to deploy and activate sensors smartly to reduce the volume of unwanted data. Mobile sensors alleviate this problem to a certain extent. This work explores the problem of optimal motion for a mobile camera to track a mobile intruder in an environment contain obstacles. When both the observer and the target are mobile and cooperative, the task is to plan the motion of both entities among obstacles. However, when they are completely non-cooperative, a pursuit-evasion game between the observer and the target arises \cite{LaValle2001}. 

In this work, we consider a pursuit-evasion game, particularly, a two-person zero-sum game in an environment with obstacles. \cite{Chung2011} provides an extensive survey of pursuit-evasion games in mobile robotic applications. We only mention a few that are related to our problem. In \cite{LaValle1997}, the problem of planing motion for a robot to maintain visibility in a cluttered environment is first introduced. The authors propose algorithms for a predictable as well as unpredictable target. While numerical solution can be obtained for a predictable target, there is no guarantee for tracking when the target is unpredictable. For practical consideration, the problem is extended to a pursuer with limited sensing range in \cite{Murrieta-Cid2007}. In \cite{Bhattacharya2011}, the authors present sufficient conditions for an observer to track the target for infinite time. An optimal tracking strategy for the observer which considers the worst case scenario caused by the evader is obtained for certain initial configurations in \cite{Zou2016}. In this paper, we present a complete solution to the pursuer's optimal strategies for all initial configurations in a worst case scenario. 

Optimal control theory has been extensively applied in motion planning of mobile robots. Minimal length paths and time-optimal trajectories have be obtained for robots with different dynamic and kinematic configurations. For example, in \cite{Balkcom2002}, the time optimal trajectories for DDRs with bounded velocity are presented. The primitives of minimum wheel-rotation paths for DDRs are presented in \cite{Chitsaz2009}. Optimal paths and velocity profiles for car-like robots which minimize the energy consumption is presented in \cite{tokekar2014energy}. Many applications relevant to target-tracking/pursuit evasion include: vision-based time-optimal strategy for a differential-drive pursuer to capture an evader \cite{Ruiz2013, Jacobo2015}; optimal strategy for the pursuer to maintain a constant distance with the evader at minimal velocity \cite{Murrieta-Cid2011}; time-optimal primitives for a pursuit evasion game between an omni-directional agent and a DDR in which the two agents can switch roles. However, these works are limited to an obstacle-free environment. Finding optimal trajectories for a robot to a fixed point has been proved to be a challenging problem in the presence of obstacles \cite{Jacobs1993, Boissonnat1996, Lavalle2000}, not to mention the problem of finding the optimal tracking strategy of a pursuer in such environment. Therefore, research on pursuit evasion games currently are limited to environment with simple obstacles which can potentially shed light on solutions in more complicated environments. For example, in \cite{nikhil09iros}, the authors present optimal solutions to the lion and man game around a circular obstacle. Visibility-based pursuit-evasion game between two holonomic agents in an environment with obstacles are investigated in \cite{Bhattacharya2010, Bhattacharya2011, Bhattacharya2016}. In \cite{Bhattacharya2010}, the authors present local necessary and sufficient conditions for surveillance and escape near termination. In \cite{Bhattacharya2016}, the problem is investigated in an environment with a circular obstacle. 

In this paper, we consider a visibility-based pursuit-evasion game between two holonomic agents in a general environment with polygonal obstacles. The main contributions of this work are as follows: First, we present the complete solution to the pursuer's optimal tracking strategies in an environment with a semi-infinite obstacle. The solution is obtained by solving a optimal control problem with state inequality constraint. The optimal strategy in this simple environment provides fundamental understanding of the problem and a building block for solutions in general environment. Second, partitions of the workspace based on tracking time and strategies are presented. Third, we extend the solutions to general environments containing multiple obstacles to generate ``pursuit fields'' \cite{} to guide the pursuer's motion.

The rest of the paper is organized as follows. In Section \ref{sec:statement}, we present the formulation of the target-tracking problem. In Section \ref{sec:corner}, we formulate and solve the tracking problem around a corner as an optimal control problem with state inequality constraints. In Section \ref{sec:general}, we present the generation of pursuit field and simulation results. Finally, we conclude in Section \ref{sec:conclusion}.

\section{Problem Statement}
\label{sec:statement}

In this section, we present the formulation of the target-tracking problem studied in this paper. Consider a planar environment containing multiple polygonal obstacles. Two mobile agents, an observer and a target, are present on the plane. We assume they all have an omni-directional field-of-view (FOV) with infinite range. They are visible to each other when the line joining them (line of sight (LOS)) does not intersect with the obstacle. Assuming that the target is initially visible to one observer. The observer's objective is to maintain a LOS with the target for the maximum possible time, while the target's objective is to break LOS in the minimum time. So what should be the optimal strategies for both agents to achieve their goals?

To tackle this problem, we first solve it in smaller scale with less complexity: the target-tracking problem in a simple environment with one observer and one target. Then, we extends the result to a general environment with multiple polygonal obstacles. Since the two agents in the problem have exactly opposite objectives and they are mobile, a pursuit-evasion game arises. Hereafter, the observer will be called the {\it pursuer}, and the target will be called the {\it evader}. Both agents are assumed to be volumeless holonomic vehicles in this work as a starting point for our future analysis for non-holonomic vehicles.

\section{target-tracking Around a Corner}
\label{sec:corner}

In this section, we present the optimal strategy for a holonomic pursuer to track a holonomic evader around a semi-infinite corner. The proposed strategy is optimal in a sense that it provides a maximum guaranteed tracking time for the pursuer for any evader strategy. Let $\mathcal{C} = \mathbb{R}^2$ be the configuration space. Let $\mathcal{C}_{obs} \in \mathcal{C}$ be the obstacle region. Thus, the free configuration space is defined as $\mathcal{C}_{free} = \mathcal{C}\backslash \mathcal{C}_{obs}$. Define positions of the pursuer and evader as $p(t) = (x_p(t), y_p(t))$ and $e(t) = (x_e(t), y_e(t))$, respectively. Since the pursuer wants to maximize the tracking time, and the evader wants to minimize it, the problem is also a two-person zero-sum game where the payoff is the tracking time.

In Figure \ref{fig:corner}, the dashed region represents the semi-infinite corner whose vertex is located at the origin. The region opposite to the corner is called \textit{star region} denoted by $S^*$. If $p(t) \in S^*$, the pursuer can track the evader forever since the entire free space $\mathcal{C}_{free}$ is visible to it. The evader wins if 1) it reaches the origin $O$ before the pursuer reaches $S^*$, or 2) it breaks LOS with the pursuer on some line $l_f$ as shown in the figure.

\begin{figure}[thbp]
	\centering
	\includegraphics[width=0.65\linewidth]{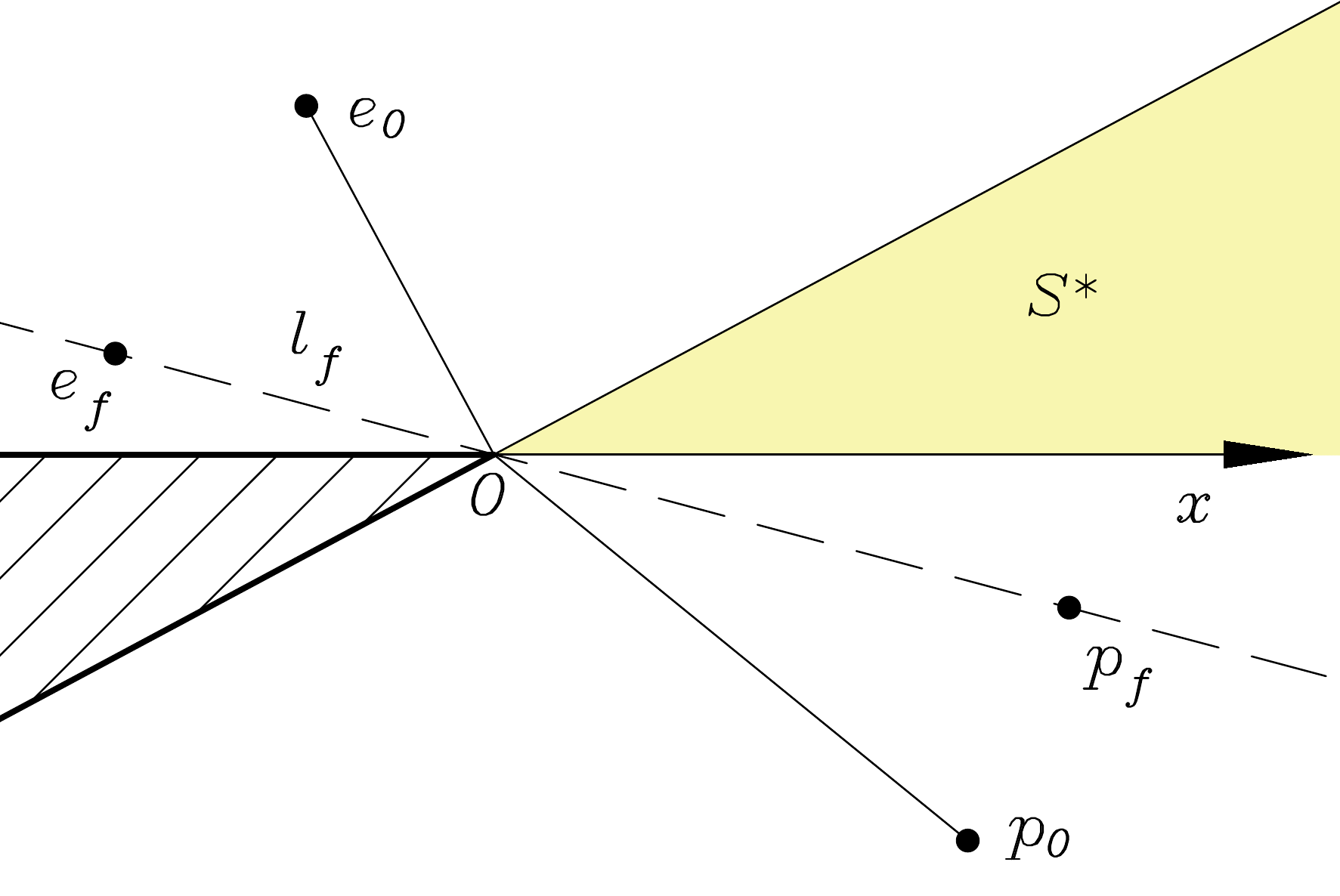}
	\caption{Tracking environment around a corner}
	\label{fig:corner}
	\vspace{-5 mm}
\end{figure}

\subsection{Interpretation as Optimal Control Problem}

We formulate the two agents target-tracking problem around a corner as an optimal control problem. First, we consider the evader's reachable set $R(t)$ which is a disc of increasing radius centered at the initial position of the evader. In order to maintain a LOS with the evader for the maximum possible time given any evader strategy, the pursuer must keep the entire disc in its FOV for the maximum possible time. The objective of the pursuer is to minimize the performance index
\begin{equation}
	J = \int_{0}^{t_f} -1 dt,
\end{equation}
where $t_f$ denotes the first time at which the evader, pursuer and origin are collinear, and the evader has a larger angular speed than the pursuer.

\begin{figure}[thbp]
	\centering
	\includegraphics[width=0.8\linewidth]{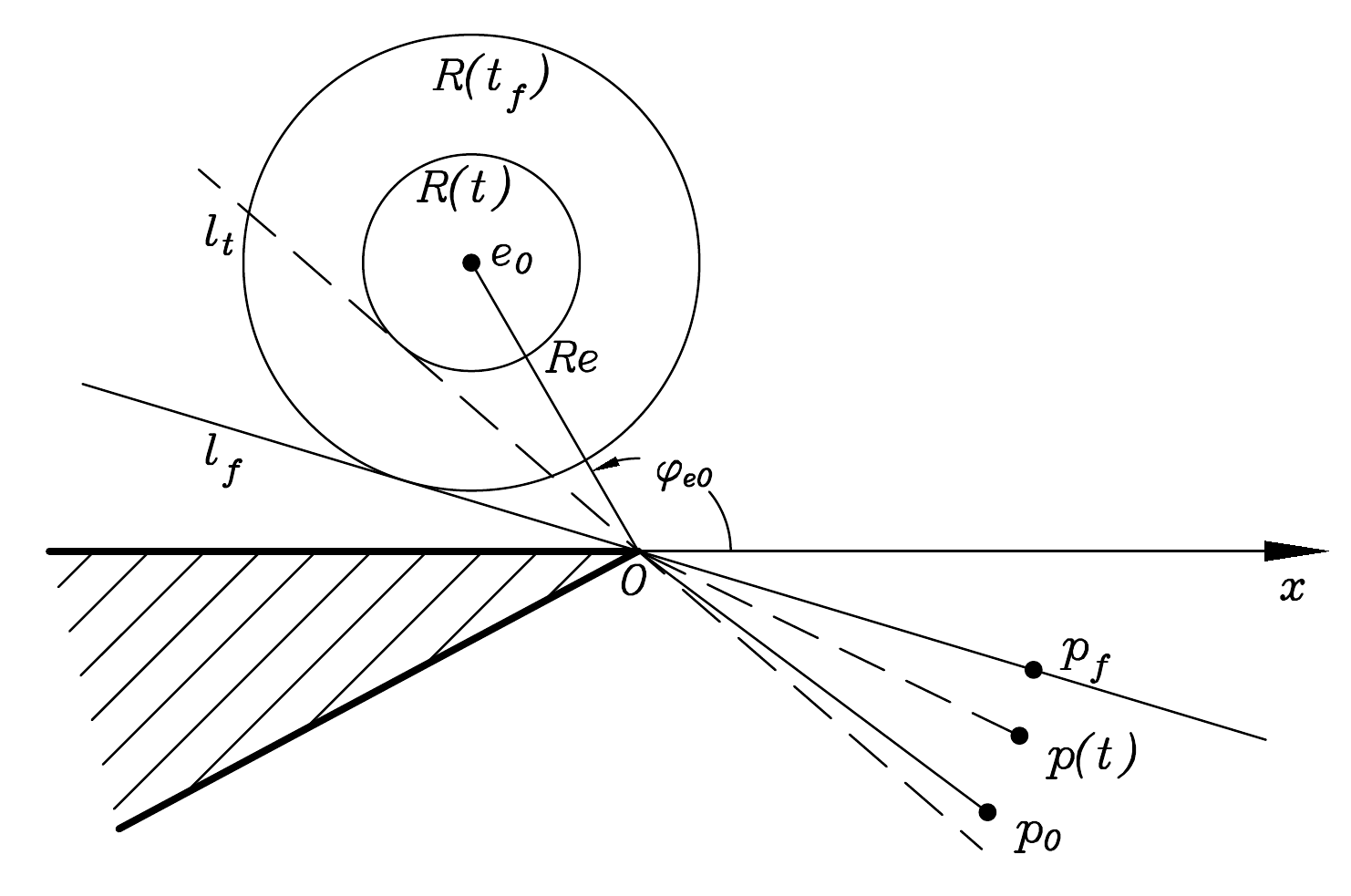}
	\caption{Tracking environment around a corner}
	\label{fig:reachable_set}
\end{figure}

Figure \ref{fig:reachable_set} illustrates the geometry of the problem. Let $v_p(t)$ and $v_e(t)$ denote the velocity of the pursuer and evader whose magnitudes are bounded by $\bar{v}_p$ and $\bar{v}_e$, respectively, and $a=\frac{\bar{v}_e}{\bar{v}_p}$ denote the ratio of their maximum speeds. The position of the pursuer, $p(t)$, is expressed in polar coordinates as $(\phi_{p}(t),r_{p}(t))$. Let $x_1(t)$ represent the angle that $l_t$ (tangent to $R(t)$ from O) subtends with the positive $x$-axis. Let $x_2(t)$ and $x_{3}(t)$ represent $r_{p}(t)$ and $\phi_{p}(t)$, respectively. From the dynamics of the players, we obtain the following state equations:

\begin{align}
	\label{eqn:states}
	\dot x_1 &= \frac{\bar{v}_e}{\sqrt{R_e^2 - (\bar{v}_e t)^2}} \\
	\dot x_2 &= -u_2 \sin u_1 \\
	\dot x_3 &= \frac{u_2}{x_2} \cos u_1,
\end{align}
where $R_e = r_e(0)$, $u_1(t) \in \mathbb{R}$ is the angle between $v_{p}(t)$ and the radial line from the origin to the pursuer's current position, and $u_2(t) \in [0, \bar{v}_p(t)]$ is the magnitude of $v_p(t)$. We assume the controls are piecewise continuous. Let $\phi_e(0)=x_1(0),\quad R_{p}=x_2(0)$ and $\phi_p(0)=x_3(0)$ represent the initial state conditions. Let the dynamics of the system be represented as $\dot{\mathbf{x}}(t) = f(\mathbf{x}, \mathbf{u}, t)$, where $\mathbf{x} = \{x_1, x_2, x_3\}$ and $\mathbf{u} = \{u_1, u_2\}$ represent the state and control vector, respectively. At $t_f$, the pursuer lies on $l_t$, and the angular speed of $l_t$ is greater than that of the pursuer. Therefore, we have the following terminal conditions:
\begin{align}
	g_1(\mathbf{x}(t_f), t_f) &= x_1(t_f) - x_3(t_f) - \pi = 0 \label{eqn:terminal_1}\\
	g_2(\mathbf{x}(t_f), t_f) &= \dot{x}_1(t_f) - \dot{x}_3(t_f) \nonumber\\
	&= \frac{\bar{v}_e}{\sqrt{R_e^2 - (\bar{v}_e t_f)^2}} - \frac{u_2(t_f)}{x_2(t_f)} \cos u_1(t_f) > 0 \label{eqn:terminal_2}
\end{align}
In order to maintain visibility before termination, the following state inequality constraint must be satisfied:
\begin{equation}
	\label{eqn:state_ineq}
	S(\mathbf{x}(t)) = \pi - x_1(t) + x_3(t) \geq 0
\end{equation}
According to \cite{Bryson1975applied}, we have a first-order state inequality constraint with the following derivative with respect to time:
\begin{equation}
	\label{eqn:1st_state}
	S^{(1)}(t) = \frac{u_2(t)}{x_2(t)} \cos u_1(t) - \frac{\bar{v}_e}{\sqrt{R_e^2 - (\bar{v}_e t)^2}} .
\end{equation}
The Hamiltonian for the constrained system is defined as follows: 
\begin{align}
	&H(\mathbf{x}(t), \mathbf{u}(t), \mathbf{p}(t), \mu(t), t) = -1 + \mathbf{p}^T f +\mu S^{(1)} \nonumber \\
	=& -1 + (p_1 - \mu) \frac{\bar{v}_e}{\sqrt{R_e^2 - (\bar{v}_e t)^2}} \nonumber\\
	&+ u_2(-p_2 \sin u_1 + \frac{p_3 + \mu}{x_2} \cos u_1) ,
	\label{eqn:ham}
\end{align}
where $\mathbf{p} = \{p_1, p_2, p_3\}$ are costates, and $\mu(t) \geq 0$ is the influence function defined as follows:
\begin{align}
	\left\{ {\begin{array}{*{20}{l}}
			{\mu  = 0, S>0 \; \mbox{the constraint boundary is inactive}}\\
			{\mu  > 0, S=0, S^{(1)} = 0\; \mbox{the constraint boundary is active}}.
		\end{array}} \right.
	\end{align}
	Since $H$ is linear in $u_2$, $u_2 \geq 0$, and $-p_2 \sin u_1 + \frac{p_3+\mu}{x_2} \cos u_1$ can always be minimized to a negative value by selecting appropriate value of $u_1$, $u_2$ should always be at its maximum value to minimize $H$. This leads to $u^*_2(t) = \bar{v}_p$. 
	
	\subsection{Inactive Boundary Constraints}
	\label{subsec:inactive}
	The point on the optimal trajectory when the boundary constraint becomes active for the first time is called the junction point. However, its existence cannot be determined a priori. Therefore, we first analyze the problem till the junction point (when $\mu = 0$) which leads to the following Hamiltonian function: 
	\begin{align}
		H &= -1 + \mathbf{p}^T f \nonumber \\
		&= -1 + p_1 \frac{\bar{v}_e}{\sqrt{R_e^2 - (\bar{v}_e t)^2}} - p_2 \bar{v}_p \sin u_1 + p_3 \frac{\bar{v}_p}{x_2} \cos u_1, \label{eqn:Hamiltonian_free}
	\end{align}
	where $u_2$ in (\ref{eqn:ham}) has been replaced with $\bar{v}_p$ in the above equation. 
	
	Since $u_1$ is unconstrained, the following optimality condition holds
	\begin{equation}
		\frac{\partial H}{\partial u_1}|_{u_1 = u_1^*(t)} = 0
		\implies \tan u^*_1 = -\frac{p^*_2}{p^*_3}x^*_2. \label{eqn:first_order_condition}
	\end{equation}
	
	Next, we show the pursuer's trajectory to minimize $J$ is a line segment. Converting the pursuer's position to Cartesian coordinates, we obtain
	\begin{equation}
		x_p(t) = r_p(t) \cos \phi_p(t) ,\;
		y_p(t) = r_p(t) \sin \phi_p(t).
	\end{equation}
	Taking the time derivative and plugging in the last two state equations in (\ref{eqn:states}) leads to the following dynamics for the pursuer in Cartesian coordinates:
	\begin{equation}
		\dot{x}_p =-\bar{v}_p \sin (u_1 + x_3),\; \dot{y}_p =  \bar{v}_p \cos(u_1 + x_3).
	\end{equation}
	Next, we show that $\dot{u}^*_1 + \dot{x}^*_3 = 0$ which implies that $u^*_1 + x^*_3$ is a constant ($\Rightarrow$ $(x^*_p(t), y^*_p(t))$ is a straight line). Differentiating both sides of (\ref{eqn:first_order_condition}) with respect to time leads to the following:
	\begin{equation}
		\dot{u}^*_1 =\bar{v}_p \frac{p_2^*}{p_3^*}\sin u^*_1 \cos^2u^*_1 - \frac{\bar{v}_p}{x^*_2}\cos^3 u^*_1.
	\end{equation} 
	Therefore, we obtain the following:
	\begin{equation*}
		\dot{u}^*_1 + \dot{x}^*_3= \frac{1}{2} \bar{v}_p \sin (2u^*_1) (\frac{p^*_2}{p_3^*}\cos u^*_1 + \frac{1}{x^*_2}\sin u^*_1) = 0.
	\end{equation*}
	The right hand side equality comes from (\ref{eqn:first_order_condition}). Therefore, $u^*_1 + x^*_3$ is a constant and $(x^*_p(t), y^*_p(t))$ lies on a straight line passing through point $(R_p \cos \phi_p(0), R_p \sin \phi_p(0))$ with slope $-\cot(u^*_1 + x^*_3)$. So the first stage of the pursuer's optimal strategy is always a line segment. 
	
	
	\subsection{Active Boundary Constraints}
	
	In this subsection, we will show when the boundary constraints become active, the pursuer's optimal trajectory has two stages. The first stage ends when the players activate the boundary constraints (\ref{eqn:state_ineq})) and (\ref{eqn:1st_state}), and the pursuer enters a second stage. In this case, the initial conditions at stage two (also the final conditions of stage 1) is such that the pursuer is on line $l_t$, and both agents have the same angular speed. The following proposition presents the pursuer's strategy in stage 2.
	
	\begin{proposition}
		\label{prop:angular_speed}
		The pursuer's optimal strategy when the boundary conditions are active is to stay on line $l_t$ and maintain the same angular speed as $l_t$ for the maximum possible time.
	\end{proposition}
	\begin{proof}
		We will prove the proposition by showing that when the pursuer is on line $l_t$, it is not possible for the pursuer to have an angular speed greater than $l_t$. In other words, when the conditions $\pi - x_1 + x_3 = 0$ is satisfied, we cannot have $\dot{x}_3 > \dot{x}_1$. Since the control variables are piecewise continuous, so is $\dot{x}_3$. At any time $t'$, when $\pi - x_1(t') + x_3(t') = 0$, there exists a $\delta t > 0$ such that
		\begin{align}
			\pi& - x_1(t'-\delta t) + \int_{t'-\delta t}^{t'} \dot{x}_1(t') dt + x_3(t'-\delta t) \nonumber \\ +& \int_{t'-\delta t}^{t'} \dot{x}_3(t') dt = 0 \\
			\pi& - x_1(t'-\delta t) + x_3(t'-\delta t) = \int_{t'-\delta t}^{t'} (\dot{x}_1(t') - \dot{x}_3(t'))dt  \label{eqn:prop1_1}\\
			\pi& - x_1(t'-\delta t) + x_3(t'-\delta t) \nonumber\\
			=& \int_{t'^- -\delta t}^{t'^-}(\dot{x}_1(t'^-) - \dot{x}_3(t'^-))dt
			\label{eqn:prop1_2}
		\end{align}
		
		From (\ref{eqn:prop1_1}), $\dot{x}_1(t') - \dot{x}_3(t')>0\Leftrightarrow \pi - x_1(t'-\delta t) + x_3(t'-\delta t) < 0$, if $\dot{x}_3$ is continuous at $t'$. This implies that LOS is broken at $t'-\delta t$. If $\dot{x}_3$ is discontinuous at $t'$, we can conclude the same from (\ref{eqn:prop1_2}) since $[t'^-,t']$ is of measure zero. Since the game terminates when the pursuer lies on $l_t$, and has a lower angular speed, the proposition holds.
	\end{proof}

	Denote the durations of the two stages of the pursuer's optimal strategy with $T_1$ and $T_2$. We present the following proposition.
	\begin{figure}[thbp]
		\centering
		\includegraphics[width=0.7\linewidth]{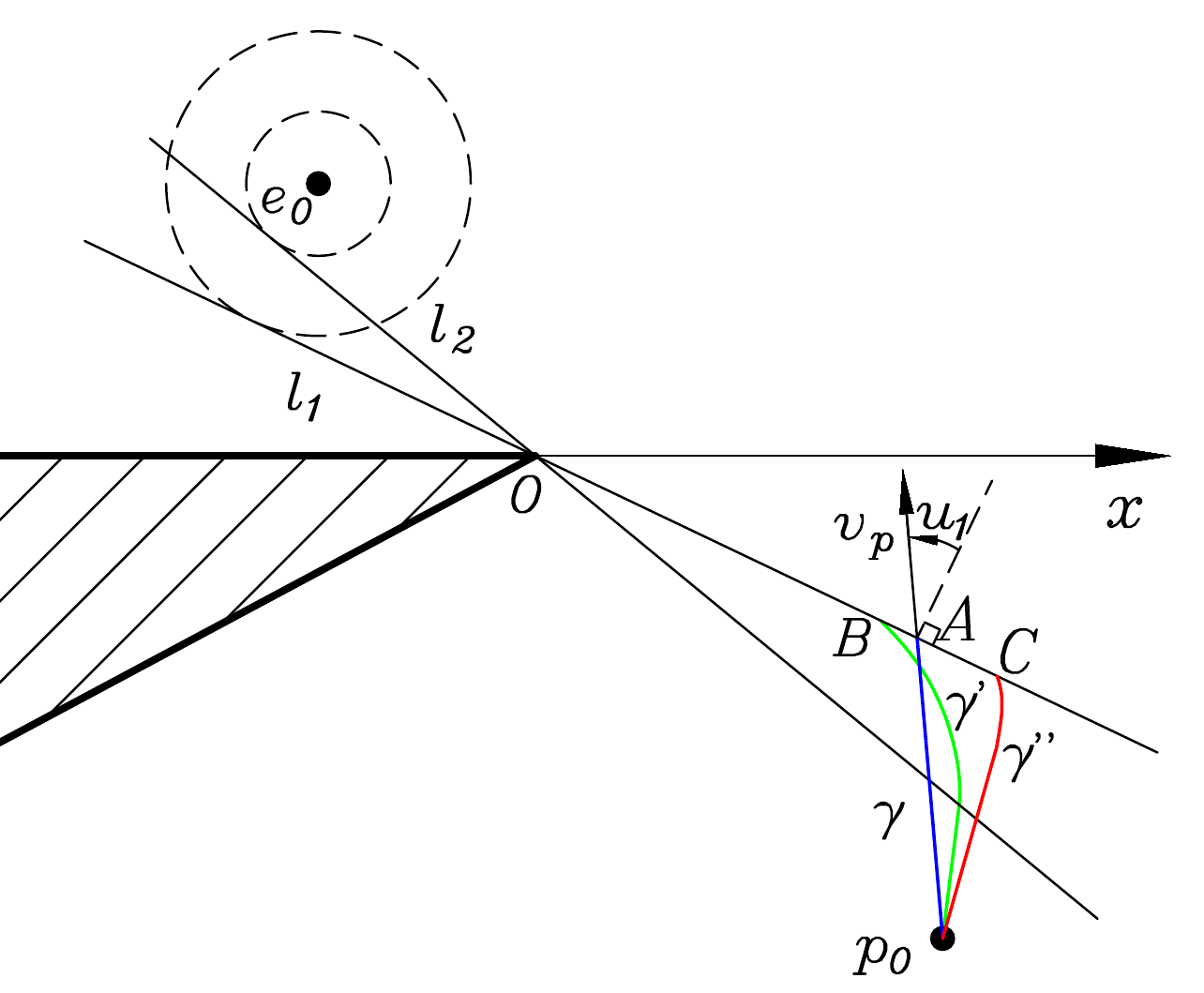}
		\caption{Trajectories of stage 1 of the pursuer.}
		\label{fig:maxT1}
	\end{figure}
	\begin{proposition}
		$\arg\max_{\bf{u}}(T_{1}+T_2)=\arg\max_{\bf{u}}T_1\quad\forall t\leq T_{1}$.
	\end{proposition}
	\begin{proof}
		According to section \ref{subsec:inactive}, the pursuer's optimal strategy to maximize tracking time before boundary constraints are active lies on a straight line which also holds for its strategy to maximize $T_1$. We will show it is also the optimal strategy for $t < T_1$ to maximize the total tracking time $T_1 + T_2$.		
		 
		We first present a property regarding the pursuer's strategy at the end of stage 1. In Figure \ref{fig:maxT1}, the blue line segment represents the pursuer's trajectory which maximizes $T_1$ in stage 1. It is denoted by $\gamma$ and ends on line $l_t$ at $A$. In this figure, $u_1 \in (0, \frac{\pi}{2})$ is the angle from the pursuer's tangential direction to its velocity direction. By the definition of $u_1$ in the problem, we know $u_1 \in [-\frac{\pi}{2}, \frac{\pi}{2}]$. If $u_1 < 0$ at $T_1$, then the pursuer's angular speed is decreasing at $T_1$. However, the angular speed of $l_t$ has been increasing in stage 1. To achieve same angular speeds at $T_1$, there must be some $\delta t > 0$, such that the pursuer's angular speed is faster than that of $l_t$ at $T_1 - \delta t$ which means visibility is broken at $T_1 - \delta t$. 
		Therefore, at the end of trajectory $\gamma$, $u_1 \in [0, \frac{\pi}{2})$. So in Figure \ref{fig:maxT1}, $\angle p_0AO \geq \frac{\pi}{2}$.
		
		The green curve in Figure \ref{fig:maxT1} denoted by $\gamma'$ represents another trajectory which finishes stage 1 with a shorter time on line $l_2$ and reaches $l_1$ at $B$ in the middle of stage 2. $B$ lies on segment $OA$. Comparing the lengths of $\gamma$ and $\gamma'$, we obtain $|\gamma'| > |\overline{p_0B}| > |\gamma|$. Since $u_2 = \bar{v}_p$ on $\gamma$, it is not possible for $\gamma'$ to reach $l_1$. Therefore, trajectory to $l_1$ other than $\gamma$ cannot reach segment $OA$ and must lie on the right of $A$ like the red trajectory $\gamma''$ in the figure. 
		
		After $\gamma$ and $\gamma''$ reach $l_1$ at $T_1$, $\gamma$ enters its stage 2 and $\gamma''$ continues in its stage 2. Both of them will maintain the same angular speed with $l_t$. Let $u_1^{\gamma}$ and $u_1^{\gamma''}$ be the directional control variables of $\gamma$ and $\gamma''$. Since $|OA| < |OC|$, by the time $u_1^{\gamma''} = 0$ and tracking ends for $\gamma''$, $u_1^{\gamma} > 0$ and tracking has not terminated for $\gamma$. Therefore, trajectory $\gamma$ is the one that maximizes the total tracking time.
	\end{proof}
	
	\subsection{Optimal trajectory of the pursuer}
	
	Based on the analysis in the previous subsections, we can conclude that the optimal trajectories for the pursuer can be divided into two classes. In the first class, the pursuer follows a straight line trajectory. In the second class, the pursuer has two stages. In the first stage, the pursuer follows a straight line trajectory till it lies on $l_t$. The second stage initiates once the pursuer lies on $l_t$, and thereafter, the pursuer's strategy is to maintain the same angular speed as $l_t$ in order to stay on it for the maximum possible time. In this subsection, we describe the pursuer's optimal trajectories from the optimal strategy.
	
	
	\vspace{0.1in}
	\noindent
	\underline{\it{Class 1}}
	
	\noindent
	We first solve the problem of finding the pursuer's trajectory to maximize the time of termination (\ref{eqn:terminal_1}) without considering any other terminal or state constraints. Thus, the Hamiltonian is defined as in (\ref{eqn:Hamiltonian_free}). With (\ref{eqn:terminal_1}) being the only terminal condition, we obtain the transversality condition that $p_2^*(t_f^*) = \alpha \frac{\partial g_1}{\partial x_2}(\mathbf{x}^*(t_f^*), t_f^*) = 0$. Combining this with (\ref{eqn:first_order_condition}) leads to $u_1^*(t_f) = 0$ which implies that the pursuer's velocity is perpendicular to the terminal line, and the optimal trajectory of the pursuer is a straight line perpendicular to the terminal line as illustrated in Figure \ref{fig:optimal}. We have shown in \cite{Zou2016} that this trajectory provides the global maximum time.
	\begin{figure}[thbp]
		\centering
		\includegraphics[width=0.6\linewidth]{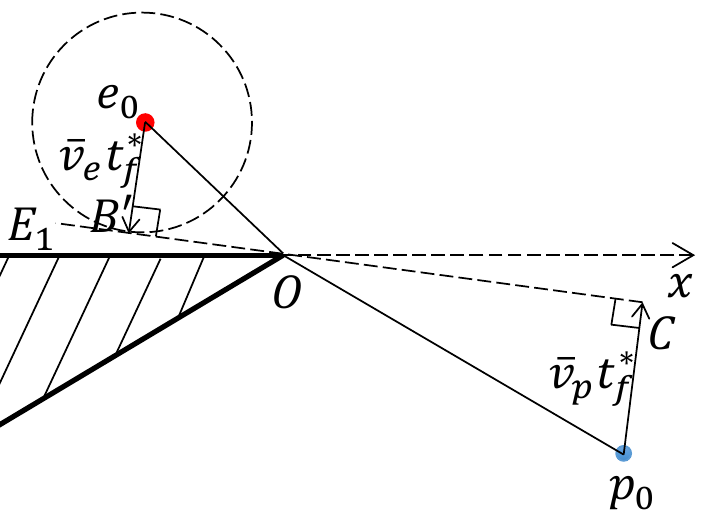}
		\caption{Pursuer's optimal trajectory of Problem 1}
		\label{fig:optimal}
		\vspace{-5 mm}
	\end{figure}
	
	Based on the geometry of the solution, it is trivial that $\frac{\bar{v}_e t_f^*}{r_e(t_f^*)} > \frac{\bar{v}_p t_f^*}{r_p(t_f^*)}$. Thus, we have $\frac{\bar{v}_e}{r_e(t_f^*)} > \frac{\bar{v}_p}{r_p(t_f^*)}$ which implies that line $l_t$ has an angular speed greater than the pursuer at termination. Therefore, solution of this problem satisfies the terminal constraint (\ref{eqn:terminal_2}). Then we check if the state inequality constraint (\ref{eqn:state_ineq}) for visibility is maintained for this trajectory. If the constraint is inactive till termination, we have a Class 1 solution, and the pursuer's optimal strategy is to move perpendicular to the terminal line. Otherwise, we need to check the Class 2 solution. 
	
	\vspace{0.1in}
	\noindent
	\underline{\it{Class 2}}
	
	\noindent

	Given the initial positions of the pursuer and evader, we cannot determine whether the solution lies in Class 1 or Class 2 a priori. If the Class 1 solution cannot maintain visibility, we need to solve for the solutions to the two stages of the Class 2 solution.	
		
	Next, we solve the problem of finding the pursuer's trajectory to maximize $T_1$ for stage 1. 
	Figure \ref{fig:findT1} shows the trajectory of the pursuer in stage 1 which ends on line $l_t$ at $p_T$. Our goal is to find the tracking time $T$ corresponding to the trajectory. Angles to be used are denoted in Figure \ref{fig:findT1}. 
	\begin{align}
		\sin \alpha_T = \frac{\bar{v}_e T}{R_e} \label{eqn:sin_a}\\
		\beta_T = \alpha_T - \pi + \Delta \phi_0 \label{eqn:beta_T}
	\end{align}
	By the law of sines, we obtain:
	\begin{equation}
		\label{eqn:sine_law}
		\frac{|Op_T|}{\sin(\pi-\beta_T-\gamma_T)} = \frac{\bar{v}_p T}{\sin \beta_T} = \frac{R_p}{\sin\gamma_T}
	\end{equation}
	Applying the fact that the pursuer and $l_t$ have the same angular speed at $T$, we obtain the following equation:
	\begin{equation}
		\label{eqn:ang_speed}
		\frac{\bar{v}_e}{\sqrt{R_e^2 - (\bar{v}_e t)^2}} = \frac{\bar{v}_p \sin(\pi - \gamma_T)}{|Op_T|}
	\end{equation} 
	Combining (\ref{eqn:sin_a}), (\ref{eqn:beta_T}), (\ref{eqn:sine_law}) and (\ref{eqn:ang_speed}) with some manipulation, we will be able to obtain an equation with only one unknown variable $T$. Due to trigonometric equations, the resulting equation of $T$ is transcendental in nature and very complicated. We will not present the analytic form of it here. But the unknown variable $T$ can be solved easily from above equations using any numerical solver. Note that there may be more than one $T$ that satisfy the final equation. In this case, we need to check the visibility along the pursuer's trajectory corresponding to each $T$. The maximum $T$ that maintains visibility of evader's reachable disk and its corresponding trajectory is the solution. 
	\begin{figure}[thbp]
		\centering
		\includegraphics[width=0.7\linewidth]{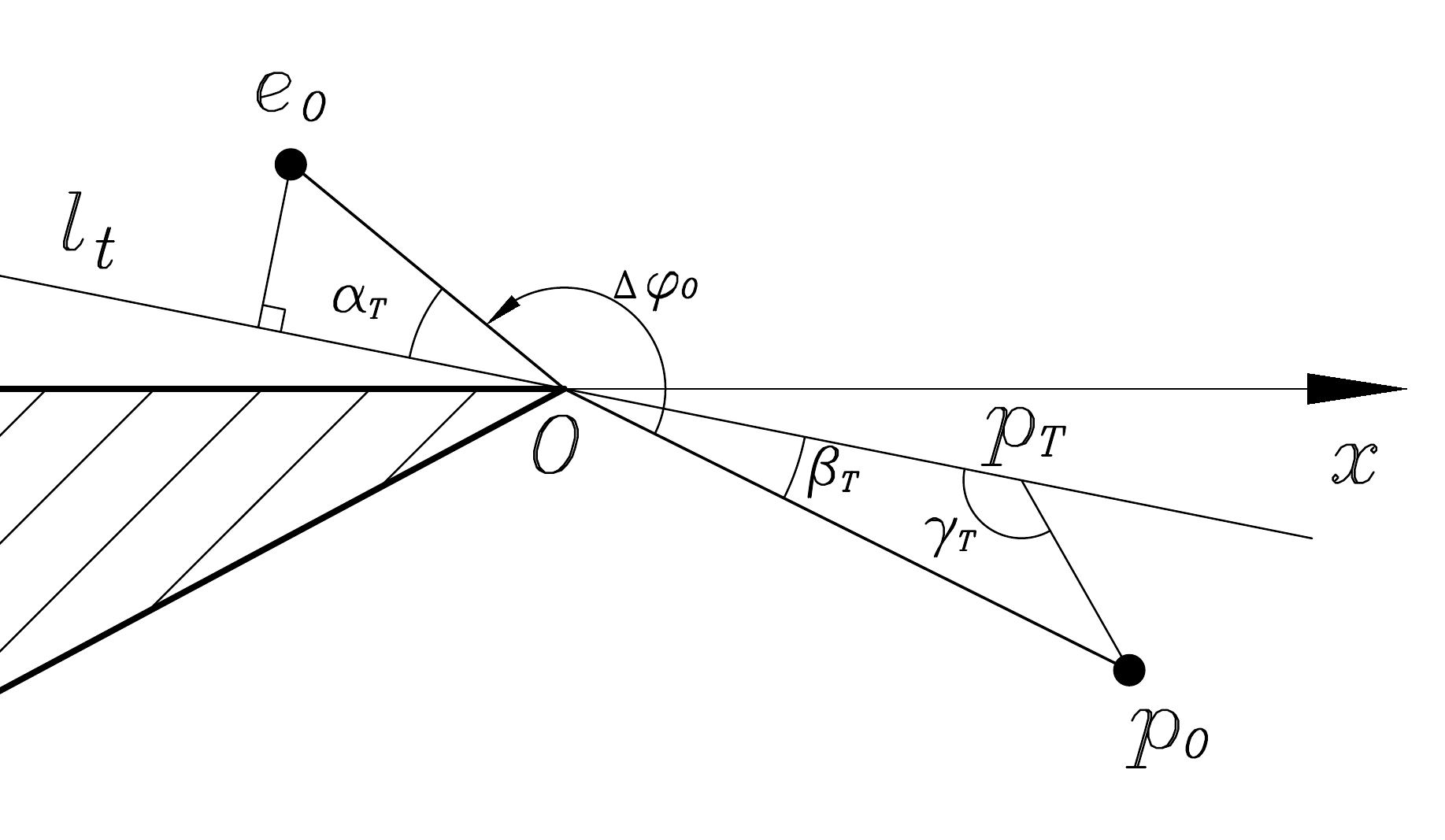}
		\caption{Geometry of stage 1 of Case 2.}
		\label{fig:findT1}
		\vspace{-5 mm}
	\end{figure}
	
	Finally, we solve the problem of finding the pursuer's strategy to maximize $T_2$ during stage 2. Based on the previous discussion, we can conclude that the pursuer will approach the corner while maintain the same angular speed with $l_t$ to stay on it during stage 2. By the end of this stage, the pursuer's velocity is aligned with its tangential direction, since that is how the pursuer achieves its maximum angular speed at the last moment. Since the motion of $l_t$ is just a function of time, we can determine the corresponding strategy of the pursuer in this stage.
	
	Assuming stage 1 of Class 2 terminates at $T_1^*$, for $t > T_1^*$, the pursuer's motion is described as follows:
	\begin{equation}
		\dot{x}_3 = \frac{\bar{v}_e}{\sqrt{R_e^2 - (\bar{v}_e t)^2}},
	\end{equation}
	with $u_2(t) = \bar{v}_p$ and $x_3(T_1^*) = x_1(T_1^*) - \pi$. Figure \ref{fig:case2path} illustrates an example trajectory of the pursuer of Case 2. The red curve represents the trajectory of the point of tangency of the evader's reachable disk. The blue curve represents the trajectory of the pursuer that maximizes its tracking time. Stage 1 ends on line $l_{T_1^*}$ at time $T_1^*$. Line $l_f$ is the terminal line of tracking. The blue curve between $l_{T_1^*}$ and $l_f$ is the pursuer's trajectory in stage 2 which maintains the same angular velocity as the evader.
	
	\begin{figure}[thbp]
		\centering
		\includegraphics[width=0.8\linewidth]{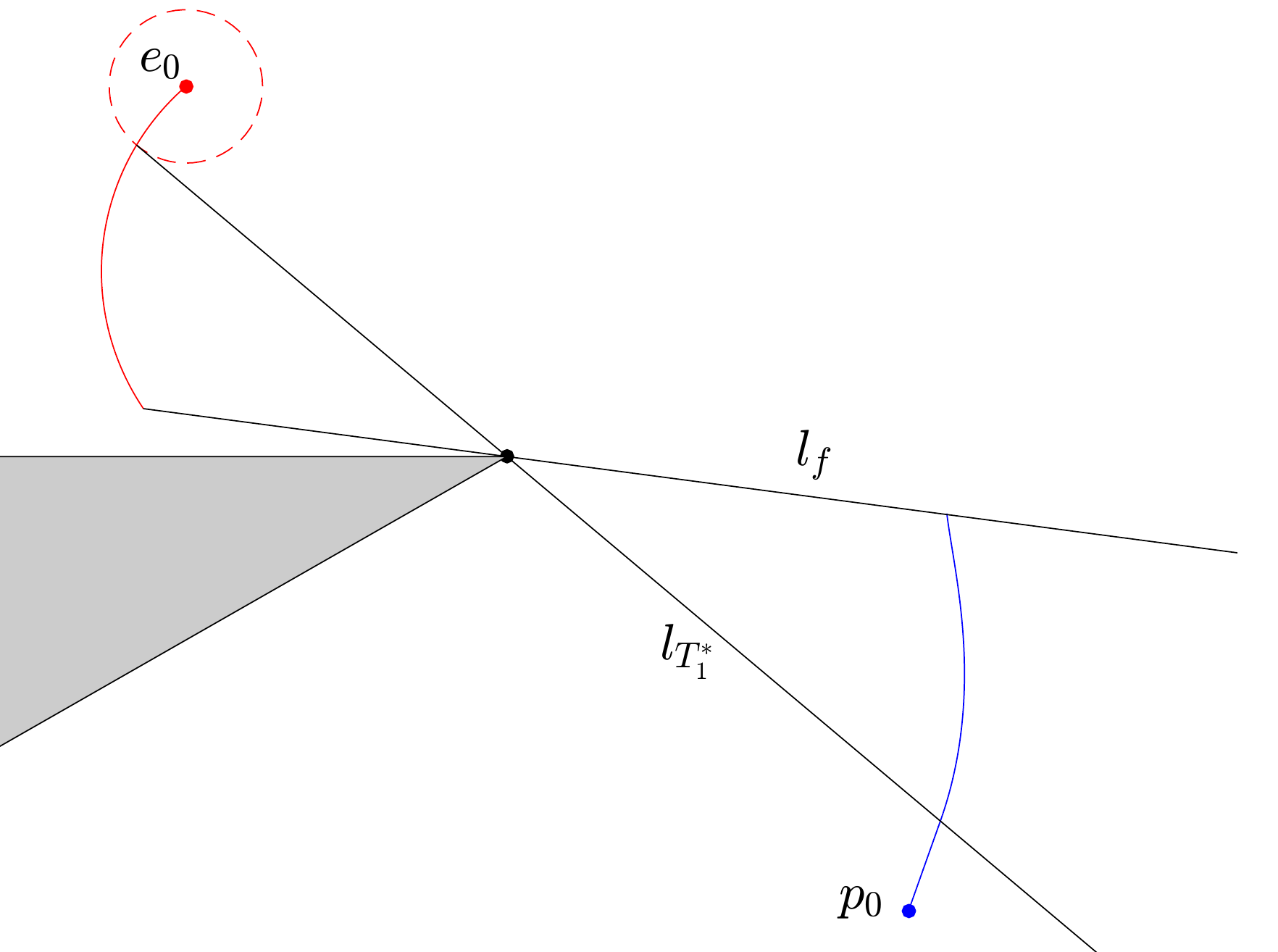}
		\caption{Example path of stage 2 of Case 2}
		\label{fig:case2path}
		\vspace{-4 mm}
	\end{figure}
	
	\subsection{Evader-based partition}
	
	In this subsection, we will present a partition of the workspace based on results in the previous subsections such that given an initial position of the evader, the partition clearly presents the regions from which the pursuer can track the evader for a specific time.
	
	To obtain such partition, we first need to obtain maximum tracking time and the corresponding strategy for any initial positions of the pursuer and evader. The optimal control problems formulated in the previous subsections do not cover all the possible scenarios for the initial positions which can be observed from Figures \ref{fig:optimal} and \ref{fig:findT1}.  Notice that in both figures, the geometry requires both agents to be located at some specific positions. Namely, they must lie on two different halves of the terminal line divided by the origin as in Figures \ref{fig:optimal} and \ref{fig:findT1}. Otherwise, the evader can only escape if it can reach the origin before the pursuer reaches the star region which leads to their optimal strategy: evader's optimal strategy is to move to the origin, and pursuer's strategy is to move to the star region. The first player to reach its destination determines the total tracking time. However, we cannot tell whether the geometry in Figures \ref{fig:optimal} and \ref{fig:findT1} can be constructed. Therefore, we have to numerically check for the existence of the solutions to Class 1 and 2. If no solution exists for both of them, it indicates the evader cannot break line of sight with the pursuer by reaching the opposite half of the terminal line from the pursuer, and the evader's optimal strategy is to move to the origin, and the pursuer's optimal strategy is to move to the star region.
	\begin{figure}[bhp]
		\centering
		\includegraphics[width=0.7\linewidth]{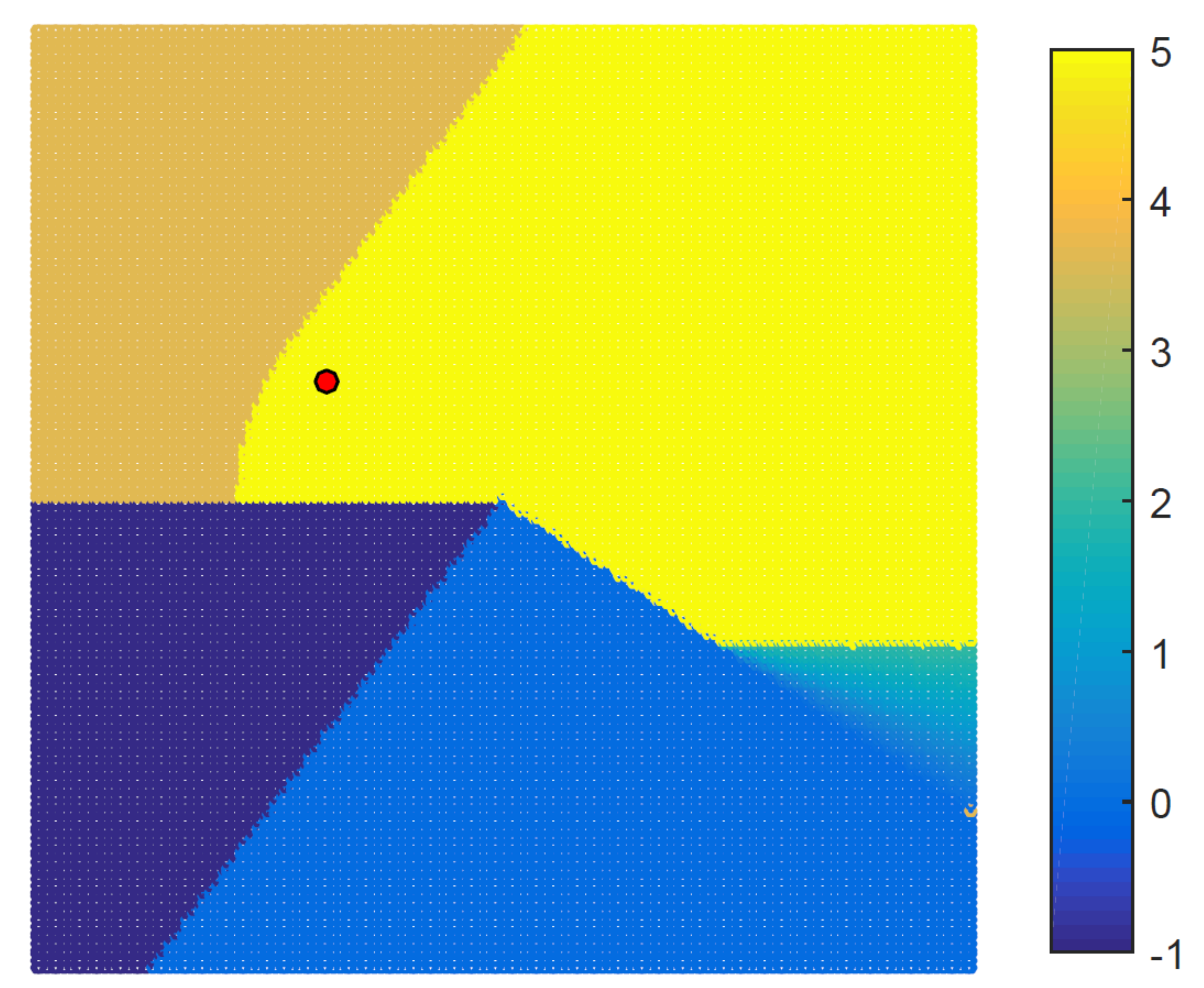}
		\caption{Evader-based partition. Red dot: Evader}
		\label{fig:e_par_holo}
		\vspace{-3 mm}
	\end{figure}
	\begin{figure}[thbp]
		\centering
		\includegraphics[width=0.6\linewidth]{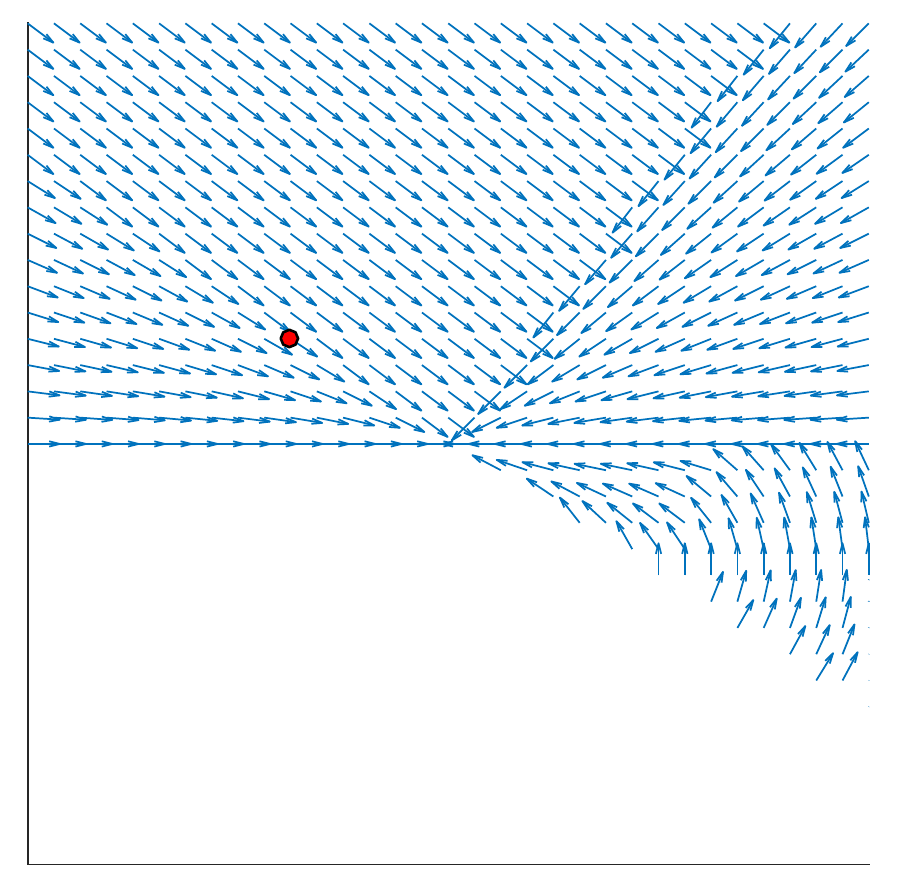}
		\caption{Evader-based vector field. Red dot: Evader}
		\label{fig:vector}
	\end{figure}
	Based on the above analysis, we compute the tracking time and strategy for all possible initial pursuer positions in the workspace given an initial evader position represented as a red dot in Figure \ref{fig:e_par_holo}. In this figure, the dark blue region represents the obstacle, and light yellow region represents pursuer win region for infinite tracking time. The optimal control for the pursuer at any given position inside the partitions generates a vector field as shown in Figure \ref{fig:vector}. The vector field originates from the pursuer's optimal trajectories from all positions at the beginning of the tracking. Note that we assign a vector towards the origin to the pursuer when it lies in the star region, though it has numerous choices to stay in the star region. Figure \ref{fig:strategy_partition} presents the evader based partition on the pursuer's optimal strategies for the maximum possible time. 

	\begin{figure}[thbp]
		\centering
		\includegraphics[width=0.9\linewidth]{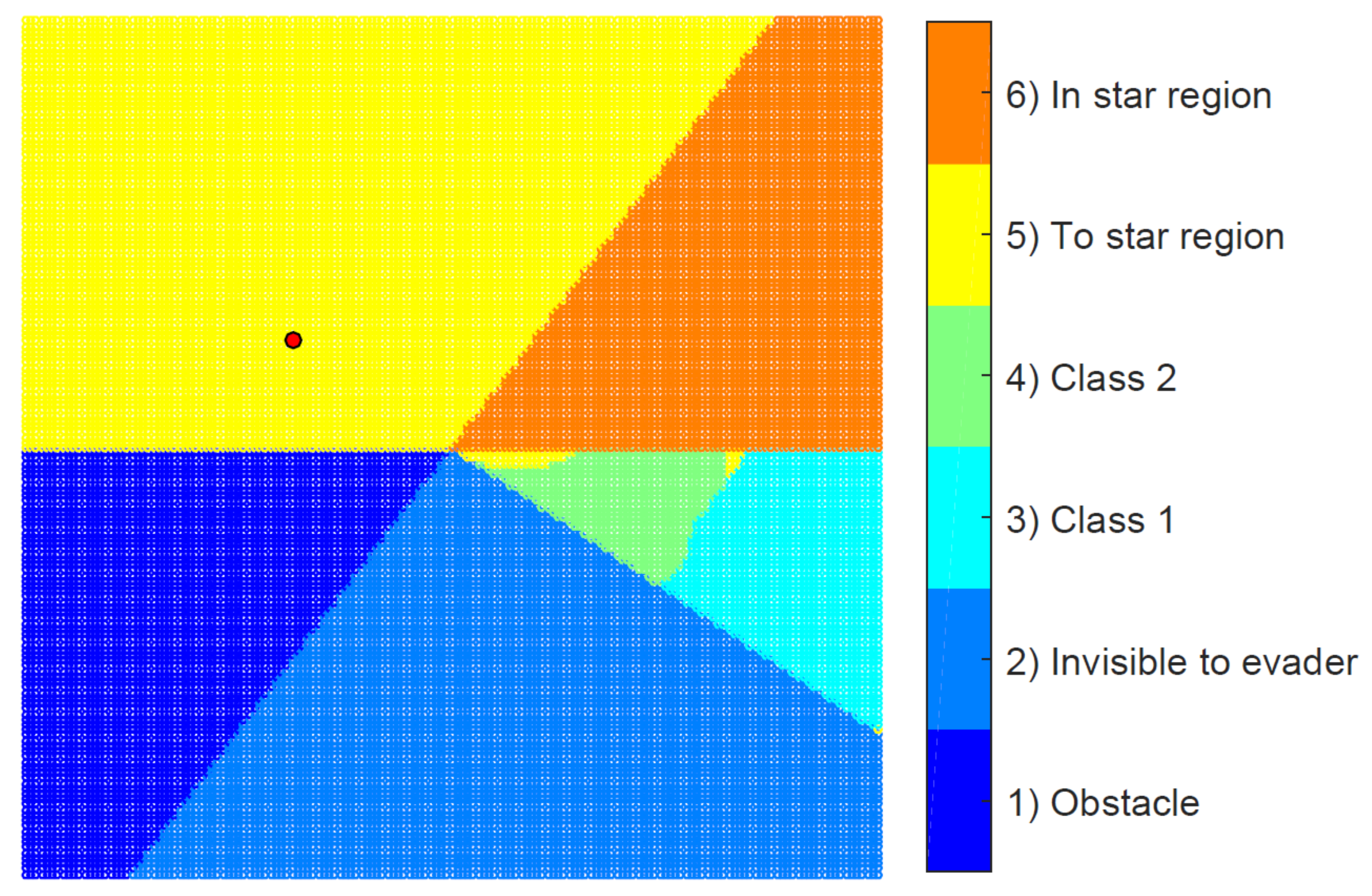}
		\caption{Evader-based pursuer strategy partition. The colors in the color bar denote: 1) pursuer can take any strategy as long as it stays in the star region; 2) pursuer takes the shortest path to star region; 3) pursuer follows the solution of Class 1; 4) pursuer follows the solution of Class 2; 5) pursuer is initially not visible from this point; 6) obstacle.}
		\label{fig:strategy_partition}
	\end{figure}

	\subsection{Pursuer-based partition}
	
	Similar to the evader-based partition, we can construct a pursuer-based partition. Given the initial position of the pursuer, we present the regions from which the evader can escape in a specific time. An analysis similar to the previous section follows in order to construct the partition. Figure \ref{fig:p_par} shows the tracking time for several initial positions of the evader for a given pursuer position.
	
	\begin{figure}[thbp]
		\centering
		\includegraphics[width=0.7\linewidth]{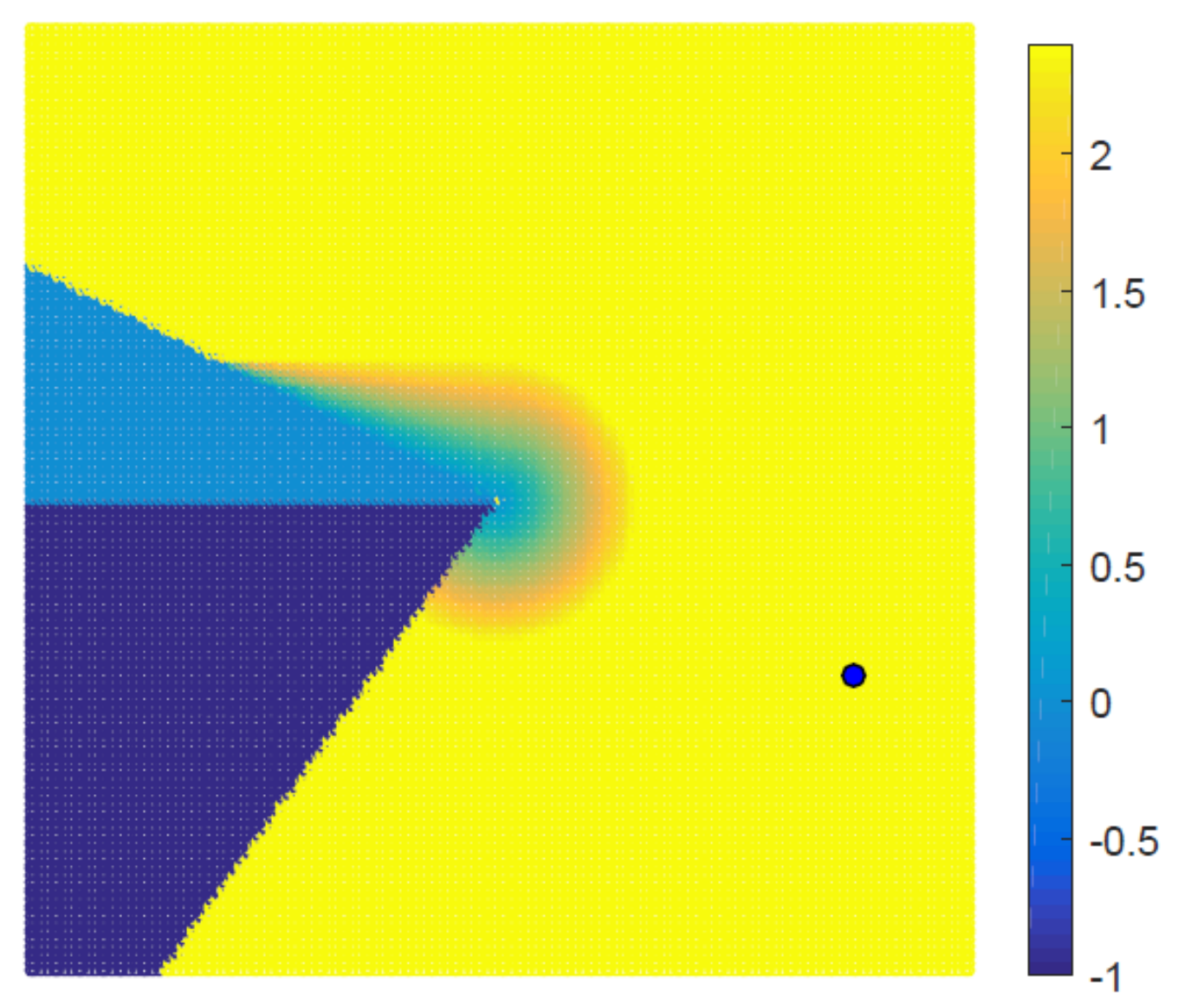}
		\caption{Pursuer-based partition. Blue dot: Pursuer. Dark blue represents the obstacle. Yellow represents the region for evader such that it can not escape in finite time.}
		\label{fig:p_par}
		\vspace{-5 mm}
	\end{figure}

\section{Target Tracking in General Environment}
\label{sec:general}

In this section, we extend the problem of target tracking in a general environment. The pursuer's tracking strategies obtained in Section \ref{sec:corner} provides the maximum possible tracking time in a simple environment with only one corner. Since the solution considers the worst case scenario, the pursuer's tracking time is guaranteed for any evader strategy. However, since the evader is completely unpredictable, uncertainty arises in both agents' strategies in the presence of multiple corners. In order to take advantage of the optimality of the strategy, we will introduce a pursuit field \cite{Mengzhe2014} in this section to guide the pursuer. The fundamental idea of this section is to design tracking strategies that maintain the optimality to the greatest extent.

We first introduce some notations and the concept of pursuit field. Let $\mathcal{V}$ be the set of all corners of polygonal obstacles in the environment. For a pursuer and an evader located at $p(t)$ and $e(t)$, $\mathcal{V}_{p(t)}$ and $\mathcal{V}_{e(t)}$ represents the set of corners that are visible to the pursuer and evader, respectively. Let $\mathcal{V}_{ref}$ denote the set of reflex vertices(corners) whose interior angles are greater than $\pi$. Since the evader cannot escape from a reflex vertex, the pursuer is only interested in the following set of vertices $\mathcal{V}^* = (\mathcal{V}\backslash \mathcal{V}_{ref}) \cap \mathcal{V}_{p(t)} \cap \mathcal{V}_{e(t)}$. 
For each vertex $\mathcal{V}^*$, a vector field as in Figure \ref{fig:vector} can be obtained in $\mathcal{C}_{free}$ given the evader's position. The pursuer's moving direction at a point is dictated by the vector at this point, and the corresponding tracking time can be computed. Denote the vector with $v_i$ and tracking time with $T_i$ for the $i$th vertex. Thus, we define a weight function $w_i(T_i)$ for each vertex which is a function of $T_i$. The weighted sum is defined as follows
\begin{equation}
v_{sum} = \sum_{\mathcal{V}^*} w_i(T_i) v_i.
\end{equation}
Normalize $v_{sum}$ to a unit vector as a guiding vector for the pursuer in the general environment. The vector field formed by $v_{sum}$ is called the \textit{pursuit field}. The weight function can be designed in different ways. We propose two potential weight functions and pursuer's corresponding strategies as follows.

\subsection{Distance-based Strategy}
\label{subsec:distance}

Consider an environment with several corners located far away from each other. If both agents are very close to one corner compared to the others, then it is reasonable for both of them to plan their strategies according to this corner. Therefore, we can consider the optimal solution from Section \ref{sec:corner} still holds locally around this corner. Following from this idea, we present the weight to be a function of the pursuer's and evader's distance to a corner. For each corner $i \in \mathcal{V}^*$, define $d_i$ as the sum of the pursuer's and evader's Euclidean distance to the corner. Then the weight function is defined as
\begin{equation}
	{w_j} = \left\{ {\begin{array}{*{20}{c}}
		{1,\;j = \argmin_{i}  ({d_i})}\\
		{0,j \ne \argmin_{i} ({d_i})}
		\end{array}} \right.. \label{eqn:distance}
\end{equation}
So the pursuit field is dominated by the vector field generated around the corner with the shortest total distance to the pursuer and the evader. For example, in Figure \ref{fig:distance}, the dominant corner is connected to the pursuer and evader by blue and red dashed lines around which the pursuit field is generated.

\begin{figure}[thbp]
	\centering
	\includegraphics[width=0.9\linewidth]{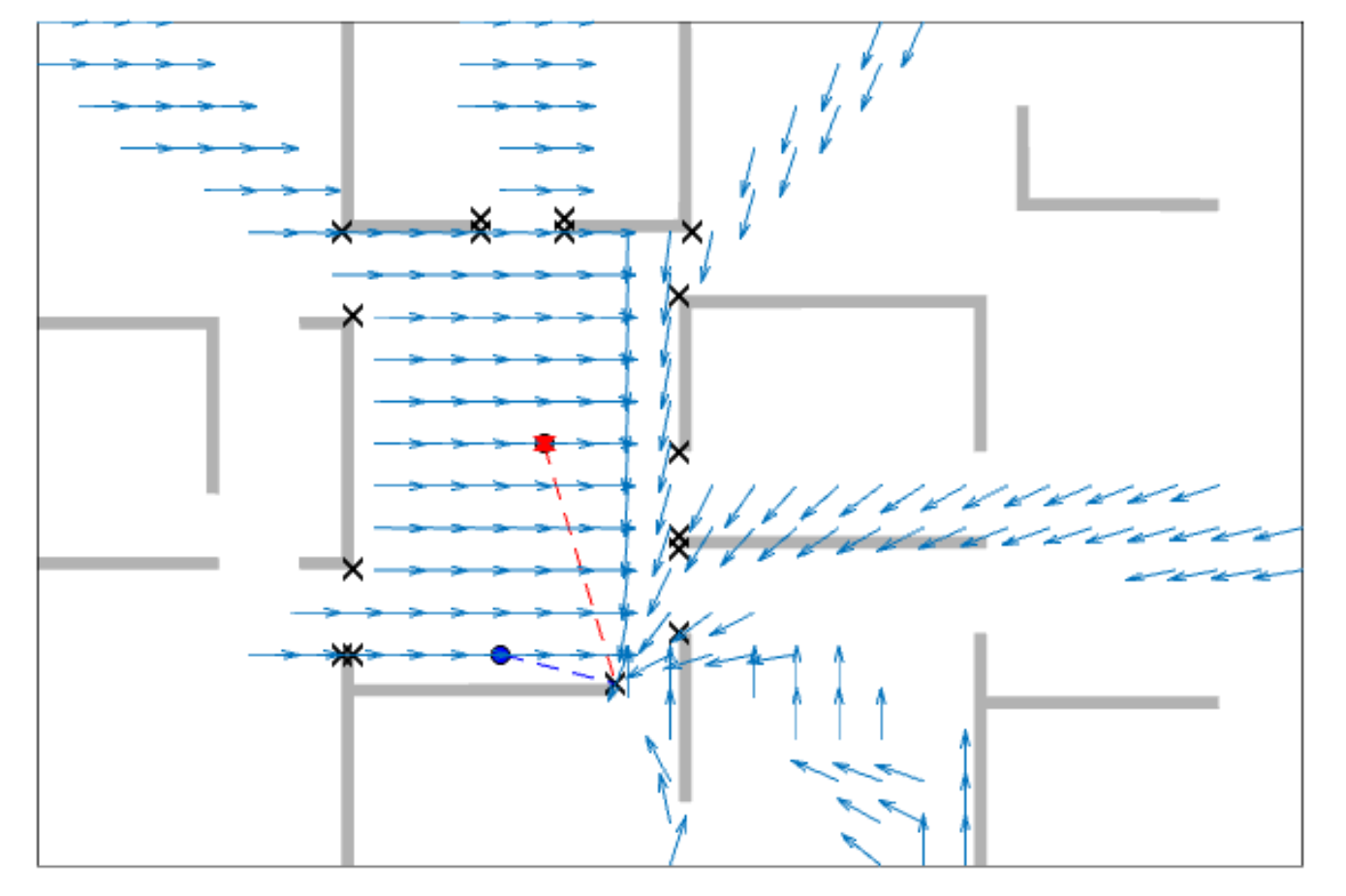}
	\caption{Pursuit field for distance-based strategy. Grey region: Obstacles. Red dot: Evader. Blue dot: Pursuer. Black crosses: corners in $\mathcal{V}^*$}
	\label{fig:distance}
	\vspace{-5 mm}
\end{figure}

\subsection{Time-Based Weighted Sum}

Since the objective of the pursuer is to track the evader for the maximum possible time, the tracking time obtained in Section \ref{sec:corner} is a candidate measure of allocating weight. Considering the fact that the evader wants to break LOS in the shortest time, it is reasonable to allocate more weight to vectors which correspond to shorter tracking time. Therefore, we can propose a feasible weight functions based on tracking time as follows.
\begin{equation}
	w_i(T_i) = \frac{1}{T_i}, \label{eqn:weighted_sum}
\end{equation}
where $T_i$ represents the pursuer's tracking time around corner $i$. Figure \ref{fig:vec1} demonstrates the pursuit field generated from weight function (\ref{eqn:weighted_sum}).
\begin{figure}[thbp]
	\centering
	\includegraphics[width=0.9\linewidth]{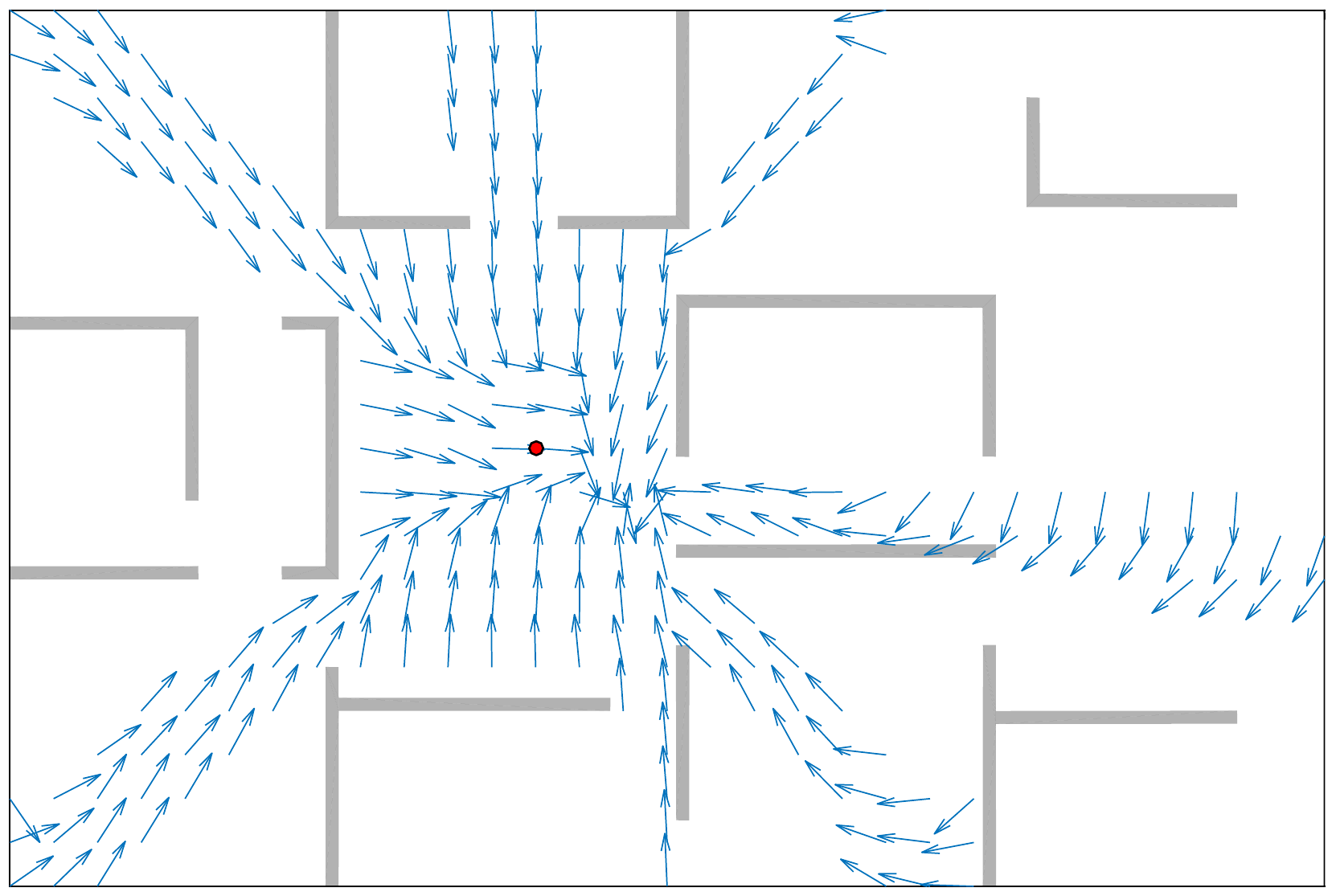}
	\caption{Pursuit field from weight function (\ref{eqn:weighted_sum}). Grey region: Obstacles. Red dot: Evader}
	\label{fig:vec1}
\end{figure}
\begin{figure}[thbp]
	\centering
	\includegraphics[width=0.9\linewidth]{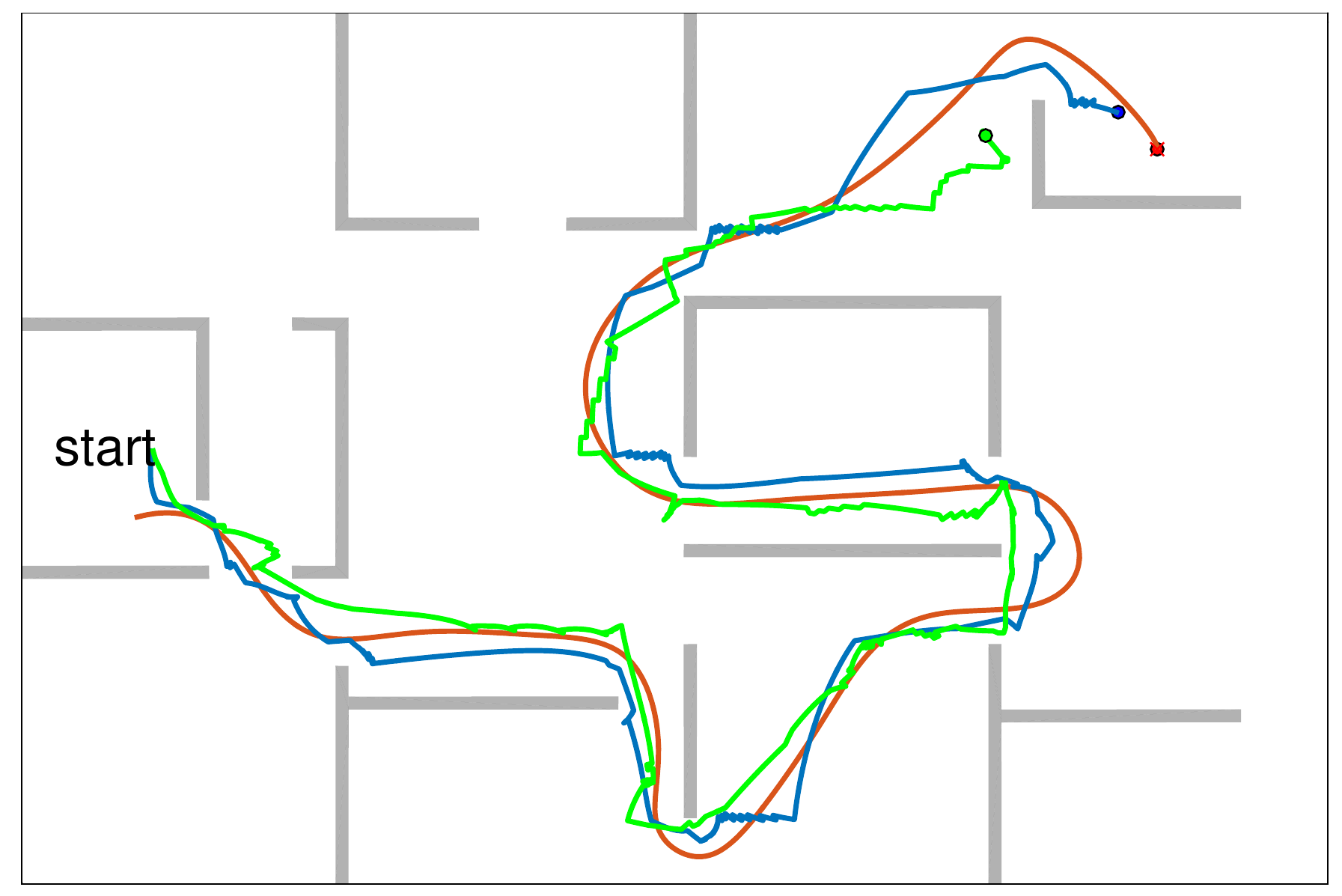}
	\caption{Trajectories of two agents. Red: evader. Green: pursuer applying time-based weighted sum of vector fields. Blue: pursuer applying distance based trajectory.}
	\label{fig:general_traj}
	\vspace{-5 mm}
\end{figure}
The aforementioned pursuit fields are constructed to guide a pursuer in simulations. In simulations, we observed that the pursuit field itself has some local constraints. Therefore, We included a vector which points towards the evader. The summation of this vector together with the pursuit field determines the pursuer' trajectory. Figure \ref{fig:general_traj} presents the trajectories of the pursuer when the evader follows a predefined path in a general environment. (The evader's path is not known to the pursuer.)

In this section, we proposed two techniques on generating tracking trajectories for the pursuer. Due to the unpredictable nature of the evader's motion and the complexity of the environment, the guaranteed tracking time in an environment with a single corner no longer holds. To the best of the authors' knowledge, the problem of finding the optimal tracking strategy of the pursuer in such an environment remains unsolved. The discussion in this section is meant to present the potential application of solutions in Section \ref{sec:corner}. Investigating various weight function and evaluating their performance is our ongoing work.

\section{Conclusion}
\label{sec:conclusion}
	
In this paper, we addressed the visibility-based target-tracking problem in a general environment. We first formulated the problem as an optimal control problem with state inequality constraint in a simple environment with one corner. We identified and analyzed two subproblems which emerge from the activation of the boundary conditions. A complete optimal tracking strategy for the pursuer which provides the maximum possible tracking time considering the worst case scenario was obtained following from the optimal control problem. To exploit the optimality of the solution in simple environment, we proposed techniques based on the optimal solutions to generate a pursuit field to guide the pursuer in a general polygonal environment. 

The idea of pursuit field is meant to explore potential applications of the optimal solution in simple environment. Due to the complexity of a general environment and the uncertainty of the evader's motion, guaranteed tracking has been a challenging problem. We believe the optimal strategy obtained in this paper can be a building block or primitive for future investigation of the problem. We will explore various approaches of extending the solution to improve the tracking performance in a general environment. We will also investigate the tracking problem when both agents have dynamic constraints.

\bibliographystyle{IEEEtran}
\bibliography{mybib}

\begin{thebibliography}{10}
\providecommand{\url}[1]{#1}
\csname url@rmstyle\endcsname
\providecommand{\newblock}{\relax}
\providecommand{\bibinfo}[2]{#2}
\providecommand\BIBentrySTDinterwordspacing{\spaceskip=0pt\relax}
\providecommand\BIBentryALTinterwordstretchfactor{4}
\providecommand\BIBentryALTinterwordspacing{\spaceskip=\fontdimen2\font plus
\BIBentryALTinterwordstretchfactor\fontdimen3\font minus
  \fontdimen4\font\relax}
\providecommand\BIBforeignlanguage[2]{{%
\expandafter\ifx\csname l@#1\endcsname\relax
\typeout{** WARNING: IEEEtran.bst: No hyphenation pattern has been}%
\typeout{** loaded for the language `#1'. Using the pattern for}%
\typeout{** the default language instead.}%
\else
\language=\csname l@#1\endcsname
\fi
#2}}

\bibitem{bhattacharya2014vision}
S.~Bhattacharya, G.~Warnell, R.~Chellappa, and T.~Basar, ``Vision-guided
  feedback control of a mobile robot with compressive measurements and side
  information,'' in \emph{Amer. Control Conf.}, June 2014, pp. 97--102.

\bibitem{Warnell2015}
G.~Warnell, S.~Bhattacharya, R.~Chellappa, and T.~Basar, ``Adaptive-rate
  compressive sensing using side information,'' \emph{IEEE Trans. Image
  Process.}, vol.~24, no.~11, pp. 3846--3857, Nov. 2015.

\bibitem{LaValle2001}
S.~LaValle and J.~Hinrichsen, ``Visibility-based pursuit-evasion: the case of
  curved environments,'' \emph{IEEE Trans. Robot. Autom.}, vol.~17, no.~2, pp.
  196--202, Apr. 2001.

\bibitem{Chung2011}
T.~Chung, G.~Hollinger, and V.~Isler, ``\BIBforeignlanguage{English}{Search and
  pursuit-evasion in mobile robotics},''
  \emph{\BIBforeignlanguage{English}{Auton. Robots}}, vol.~31, no.~4, pp.
  299--316, July 2011.

\bibitem{LaValle1997}
S.~LaValle, H.~Gonzalez-Banos, C.~Becker, and J.-C. Latombe, ``Motion
  strategies for maintaining visibility of a moving target,'' in \emph{Proc.
  IEEE Int. Conf. Robot. Autom.}, vol.~1, Apr. 1997, pp. 731--736 vol.1.

\bibitem{Murrieta-Cid2007}
R.~Murrieta-Cid, T.~Muppirala, A.~Sarmiento, S.~Bhattacharya, and
  S.~Hutchinson, ``Surveillance strategies for a pursuer with finite sensor
  range,'' \emph{Int. J. Robot. Res.}, vol.~26, no.~3, pp. 233--253, 2007.

\bibitem{Bhattacharya2011}
S.~Bhattacharya and S.~Hutchinson, ``{A cell decomposition approach to
  visibility-based pursuit evasion among obstacles},'' \emph{Int. J. Robot.
  Res.}, vol.~30, no.~14, pp. 1709--1727, Sept. 2011.

\bibitem{Zou2016}
R.~Zou and S.~Bhattacharya, ``Visibility-based finite-horizon target tracking
  game,'' \emph{IEEE Robot. Autom. Lett.}, vol.~1, no.~1, pp. 399--406, Jan
  2016.

\bibitem{Balkcom2002}
D.~J. Balkcom and M.~T. Mason, ``Time optimal trajectories for bounded velocity
  differential drive vehicles,'' \emph{Int. J. Robot. Res.}, vol.~21, no.~3,
  pp. 199--217, Mar. 2002.

\bibitem{Chitsaz2009}
H.~Chitsaz, S.~M. LaValle, D.~J. Balkcom, and M.~T. Mason, ``Minimum
  wheel-rotation paths for differential-drive mobile robots,'' \emph{Int. J.
  Robot. Res.}, vol.~28, no.~1, pp. 66--80, Jan. 2009.

\bibitem{tokekar2014energy}
P.~Tokekar, N.~Karnad, and V.~Isler,
  ``\BIBforeignlanguage{English}{Energy-optimal trajectory planning for
  car-like robots},'' \emph{\BIBforeignlanguage{English}{Auton. Robots}},
  vol.~37, no.~3, pp. 279--300, Oct. 2014.

\bibitem{Ruiz2013}
U.~Ruiz, R.~Murrieta-Cid, and J.~L. Marroquin, ``{Time-optimal motion
  strategies for capturing an omnidirectional evader using a differential drive
  robot},'' \emph{IEEE Trans. Robot.}, vol.~29, no.~5, pp. 1180--1196, Oct.
  2013.

\bibitem{Jacobo2015}
D.~Jacobo, U.~Ruiz, R.~Murrieta-Cid, H.~Becerra, and J.~Marroquin, ``A visual
  feedback-based time-optimal motion policy for capturing an unpredictable
  evader,'' \emph{Int. J. Control}, vol.~88, no.~4, pp. 663--681, 2015.

\bibitem{Murrieta-Cid2011}
R.~Murrieta-Cid, U.~Ruiz, J.~L. Marroquin, J.~P. Laumond, and S.~Hutchinson,
  ``{Tracking an omnidirectional evader with a differential drive robot},''
  \emph{Auton. Robots}, vol.~31, no.~4, pp. 345--366, Aug. 2011.

\bibitem{Jacobs1993}
P.~Jacobs and J.~Canny, \emph{Nonholonomic Motion Planning}.\hskip 1em plus
  0.5em minus 0.4em\relax Boston, MA: Springer US, 1993, ch. Planning Smooth
  Paths for Mobile Robots, pp. 271--342.

\bibitem{Boissonnat1996}
J.-D. Boissonnat and S.~Lazard, ``A polynomial-time algorithm for computing a
  shortest path of bounded curvature amidst moderate obstacles,'' in
  \emph{Proceedings of the Twelfth Annual Symposium on Computational
  Geometry}.\hskip 1em plus 0.5em minus 0.4em\relax New York, NY: ACM, 1996,
  pp. 242--251.

\bibitem{Lavalle2000}
S.~M. Lavalle, J.~J. Kuffner, and Jr., ``Rapidly-exploring random trees:
  Progress and prospects,'' in \emph{Algorithmic and Computational Robotics:
  New Directions}, 2000, pp. 293--308.

\bibitem{nikhil09iros}
N.~Karnad and V.~Isler, ``Lion and man game in the presence of a circular
  obstacle,'' in \emph{IEEE Int. Conf. Intell. Robot. Syst.}, Oct. 2009, pp.
  5045--5050.

\bibitem{Bhattacharya2010}
S.~Bhattacharya and S.~Hutchinson, ``On the existence of {Nash} equilibrium for
  a two-player pursuit—evasion game with visibility constraints,'' \emph{Int.
  J. Robot. Res.}, vol.~29, no.~7, pp. 831--839, June 2010.

\bibitem{Bhattacharya2016}
S.~Bhattacharya, T.~Ba{\c{s}}ar, and N.~Hovakimyan, ``A visibility-based
  pursuit-evasion game with a circular obstacle,'' \emph{Journal of
  Optimization Theory and Applications}, pp. 1--12, 2016.

\bibitem{Bryson1975applied}
A.~E. Bryson and Y.-C. Ho, \emph{Applied optimal control: optimization,
  estimation and control}.\hskip 1em plus 0.5em minus 0.4em\relax New York, NY:
  CRC Press, 1975.

\bibitem{Mengzhe2014}
M.~Zhang and S.~Bhattacharya, ``Multi-agent visibility based target tracking
  game,'' in \emph{Int. Symp. Distrib. Auton. Robot. Syst.}, Nov. 2014, pp.
  440--451.

\end{thebibliography}

\end{document}